\documentclass[12pt,leqno,letterpaper]{article}

\usepackage{amsmath,amsthm,enumerate,amssymb}
\usepackage[utf8]{inputenc}

\usepackage[T1]{fontenc}
\usepackage[english]{babel}
\usepackage{graphicx}
\usepackage{color}

\usepackage{graphicx}

\newtheorem{theorem}{Theorem}[section]
\newtheorem{proposition}[theorem]{Proposition}

\newtheorem{definition}[theorem]{Definition}

\newcommand{\tr}{{\rm Tr\hskip -0.2em}~}

\begin{document}

\title{Regular operator mappings\\
and\\
multivariate geometric means}
\author{Frank Hansen}
\date{}

\maketitle

\begin{abstract}
We introduce the notion of regular operator mappings of several variables generalising the notion of spectral function. This setting is convenient for studying maps more general than what can be obtained from the functional calculus, and it allows for Jensen type inequalities and multivariate non-commutative perspectives.

As a main application of the theory we consider geometric means of $ k $ operator variables extending the geometric mean of $ k $ commuting operators and the geometric mean of two arbitrary positive definite matrices. We propose different types of updating conditions that seems natural in many applications and prove that each of these conditions, together with a few other natural axioms, uniquely defines the geometric mean for any number of operator variables. The means defined in this way are given by explicit formulas and are computationally tractable.\footnote{An earlier version of this paper was posted in the ArXiv on March 15, 2014 with the title Geometric means of several variables.}
\\[1ex]
{\bf MSC2010} classification: 26B25; 47A64\\[1ex]
{\bf{Key words and phrases:}}  regular operator mapping; perspective; multivariate geometric mean.

\end{abstract}

\section{Introduction}

The geometric mean of two positive definite operators was introduced by Pusz and Woronowicz  \cite{kn:pusz:1975}, and their definition was soon put into the context of the axiomatic approach to operator means developed by Kubo and Ando \cite{kn:kubo:1980}. Subsequently a number of authors \cite{kn:kosaki:1984, kn:ando:2004:1, kn:moakher:2005, kn:bhatia:2006, kn:Lim:2012, kn:Lawson:2014} have suggested several ways of defining means of operators for several variables as extensions of the geometric mean of two operators.

There is no satisfactory definition of a geometric mean of several operator variables that is both computationally tractable and satisfies a number of natural conditions put forward in the influential paper by Ando, Li, and Mathias  \cite{kn:ando:2004:1}. We put the emphasis on methods to extend a geometric mean of $ k $ variables to a mean of $ k+1 $ variables, and in the process we challenge one of the requirements to a geometric mean put forward by Ando, Li, and Mathias.

The symmetry condition of a geometric mean is mathematically very appealing, but the condition makes no sense in a number of applications. If for example positive definite matrices $ A_1,A_2,\dots,A_k $ correspond to measurements made at times $ t_1<t_2<\cdots<t_k $ then there is no way of permuting the matrices since time only goes forward. It makes more sense to impose an updating condition
\begin{equation}\label{updating condition}
G_{k+1}(A_1,\dots,A_k,1)=G_k(A_1,\dots,A_k)^{k/(k+1)}
\end{equation}
when moving from a mean $ G_k $ of $ k $ variables to a mean $ G_{k+1} $ of $ k+1 $ variables. The condition corresponds to taking the geometric mean of $ k $ copies of $ G_k(A_1,\dots,A_k) $ and one copy of the unit matrix. A variant condition would be to impose the equality
\begin{equation}\label{variant updating condition}
G_{k+1}(A_1,\dots,A_k,1)=G_k(A_1^{k/(k+1)},\dots,A_k^{k/(k+1)})
\end{equation}
when updating from $ k $ to $ k+1 $ variables.
It is an easy exercise to realise that if we set $ G_1(A)=A, $ then either of the conditions (\ref{updating condition}) or (\ref{variant updating condition}) together with homogeneity uniquely  defines the geometric mean of $ k $ commuting operators.  

We furthermore prove that by setting $ G_1(A)=A $ and by demanding homogeneity and a few more natural conditions, then either of the updating conditions (\ref{updating condition}) or (\ref{variant updating condition}) leads to unique but different solutions to the problem of defining a geometric mean of $ k $ operators. 
The means defined in this way are given by explicit formulas, and they are computationally tractable. They possess all of the attractive properties associated with geometric operator means discussed in \cite{kn:ando:2004:1} with the notable exception of symmetry. 
If one emphasises either of the updating conditions (\ref{updating condition}) or (\ref{variant updating condition}) we are thus forced to abandon symmetry.

Efficient averaging techniques of positive definite matrices are important in many practical applications; for example in radar imaging, medical imaging, and the analysis of financial data.

\section{Regular operator mappings}

\subsection{Spectral functions}

Let $ B(\mathcal H) $ denote the set of bounded linear operators on a Hilbert space $ \mathcal H. $ A function $ F\colon\mathcal D\to B(\mathcal H) $ defined in a convex domain $ \mathcal D $ of self-adjoint operators in $ B(\mathcal H) $ is called a spectral function, if  it can be written on the form $ F(x)=f(x) $ for some function $ f $ defined in a real interval $ I, $ where $ f(x) $ is obtained by applying the functional calculus.

The definition contains some hidden assumptions. The domain $ \mathcal D $ should be invariant under unitary transformations and
\begin{equation}\label{unitary invariance}
F(u^*xu)=u^*F(x)u\qquad x\in\mathcal D
\end{equation}
for every unitary transformation $ u $ on $ \mathcal H. $ Furthermore, to pairs of mutually orthogonal projections $ p $ and $ q $ acting on $ \mathcal H, $ the element $ pxp+qxq $ should be in $ \mathcal D $  and the equality
\begin{equation}\label{rule for block matrices}
F(pxp+qxq)=pF(pxp)p+qF(qxq)q
\end{equation}
should hold for any $ x\in B(\mathcal H) $ such that $ pxp $ and $ qxq $ are in $ \mathcal D. $ An operator function is a spectral function if and only if (\ref{unitary invariance}) and (\ref{rule for block matrices}) are satisfied, cf. \cite{kn:davis:1957, kn:hansen:2003:2}.

The notion of spectral function is not immediately extendable to functions of several variables. However, we may consider the two properties of spectral functions noticed by C. Davis as a kind of regularity conditions, and they are readily extendable to functions of more than one variable. 

The notion of a regular map of two operator variables were studied by Effros and the author in \cite{kn:hansen:2014:2}, cf. also \cite{kn:hansen:1983}.

\begin{definition}\label{definition: regular map}
Let $ F\colon\mathcal D\to B(\mathcal H) $ be a mapping of $ k $ variables defined in a convex domain $ \mathcal D\subseteq B(\mathcal H)\times\cdots\times B(\mathcal H). $
We say that $ F $ is regular if

\begin{enumerate}[(i)]

\item The domain $ \mathcal D $ is invariant under unitary transformations of $ \mathcal H $ and
\[
F(u^*x_1 u,\dots, u^*x_ku)=u^* F(x_1,\dots,x_k) u
\]
for every $ x=(x_1,\dots,x_k)\in\mathcal D $ and every unitary $ u $ on $ \mathcal H. $

\item Let $ p $ and $ q $ be mutually orthogonal projections acting on $\mathcal H $ and take arbitrary $ k $-tuples
$ (x_1,\dots,x_k) $ and $ (y_1,\dots,y_k) $ of operators in $ B(\mathcal H) $ such that the compressed tuples
\[
(p x_1 p,\dots,p x_k p)\qquad\text{and}\qquad (q y_1 q,\dots,q y_k q)
\]
are in the domain $ \mathcal D. $ Then the $ k $-tuple of diagonal block matrices
\[
(px_1p+qy_1q,\dots, p x_k p+q y_k q)
\]
is also in the domain $ \mathcal D $ and 
\[
\begin{array}{l}
F(px_1p+qy_1q,\dots, p x_k p+q y_k q)\\[2ex]
\qquad=pF(px_1p,\dots,px_kp)p+qF(qy_1q,\dots,qy_kq)q.
\end{array}
\]
\end{enumerate}                
\end{definition}
By choosing $ q $ as the zero projection in the second condition in the above definition we 
obtain
\[
F(px_1p,\dots, p x_k p)=pF(px_1p,\dots,px_kp)p,
\]
which shows that $ F $ for any orthogonal projection $ p $ on $ \mathcal H $ may be considered as a regular operator mapping
\[
F\colon \mathcal D_p\to B(p\mathcal H),
\]
where the compressed domain
\[
\displaystyle\mathcal D_p=\{(x_1,\dots,x_k)\in \bigoplus_{m=1}^k B(p\mathcal H)\mid 
\bigl(x_1\oplus 0(1-p),\dots,x_k\oplus 0(1-p)\bigr)\in\mathcal D\}.
\] 
With this interpretation we may unambiguously calculate block matrices by the formula
\[
F\left(\begin{pmatrix}
                    x_1 & 0\\
                    0     & y_1
                    \end{pmatrix},
\dots,
\begin{pmatrix}
                    x_k & 0\\
                    0     & y_k
                    \end{pmatrix}\right)
=\begin{pmatrix}
                    F(x_1,\dots,x_k) & 0\\
                    0                         & F(y_1,\dots,y_k)
                    \end{pmatrix}
\]
which is well-known from mappings generated by the functional calculus.

\subsection{Jensen's inequality for regular operator mappings}

We consider throughout this paper the domain 
\[
\mathcal D^k=\{(A_1,\dots,A_k)\mid A_1,\dots,A_k\ge 0\}
\]
of $ k $-tuples of positive semi-definite operators acting on an infinite dimensional Hilbert space $ \mathcal H. $ It is convenient to consider an infinite dimensional Hilbert space since in this case $ \mathcal H $ is isomorphic to $ \mathcal H\oplus\mathcal H $ which allows us to use block matrix techniques without imposing dimension conditions.

\begin{theorem}
Consider a convex regular mapping
\[
F\colon\mathcal D^k\to B(\mathcal H)_\text{sa}
\]
of $ \mathcal D^k $ into self-adjoint operators acting on $ \mathcal H. $
\begin{enumerate}[(i)]

\item Let $ C $ be a contraction on $ \mathcal H. $ If $ F(0,\dots,0)\le 0 $ then the inequality
\[
F(C^*A_1C,\dots,C^*A_kC)\le C^* F(A_1,\dots,A_k)C
\]
holds for $ k $-tuples $ (A_1,\dots,A_k) $ in $ \mathcal D^k. $

\item Let $ X $ and $ Y $ be operators acting on $ \mathcal H $ with $ X^*X+Y^*Y=1. $ Then the inequality
\[
\begin{array}{l}
F\bigl(X^*A_1X+Y^*B_1Y,\dots,X^*A_kX+Y^*B_kY\Bigr)\\[2ex]
\hskip 6em\le X^* F(A_1,\dots,A_k)X +  Y^* F(B_1,\dots,B_k)Y
\end{array}
\]
holds for $ k $-tuples $ (A_1,\dots,A_k) $ and $ (B_1,\dots,B_k) $ in $ \mathcal D^k. $

\end{enumerate}
\end{theorem}

\begin{proof}
By setting $ T=(1-C^*C)^{1/2} $ and $ S=(1-CC^*)^{1/2} $ we obtain that the block matrices
\[
U=\begin{pmatrix}
                C & S\\
                T & -C^*
                \end{pmatrix}
\qquad\text{and}\qquad
V=\begin{pmatrix}
                C  & -S\\
                -T & -C^*
                \end{pmatrix}
\]
are unitary operators on $ \mathcal H\oplus\mathcal H. $ Furthermore,
\[
\frac{1}{2}U^*\begin{pmatrix}
                      A & 0\\
                      0 & 0
                      \end{pmatrix}U
+\frac{1}{2}V^*\begin{pmatrix}
                      A & 0\\
                      0 & 0
                      \end{pmatrix}V
=\begin{pmatrix}
             C^*AC & 0\\
             0          & SAS
             \end{pmatrix}
\]
for any operator $ A\in B(\mathcal H). $ By using that $ F $ is a convex regular map we obtain
\[
\begin{array}{l}
\displaystyle\begin{pmatrix}
          F(C^*A_1C,\dots,C^*A_kC) & 0\\
          0                                           & F(S A_1 S,\dots,S A_k S)
          \end{pmatrix}\\[3ex]
\displaystyle =F\left(\begin{pmatrix}
                                C^*A_1C & 0\\
                                0              & S A_1 S
                                \end{pmatrix},\dots,\begin{pmatrix}
                                                                         C^*A_kC & 0\\
                                                                         0              & S A_k S
                                                                         \end{pmatrix}\right)\\[3ex]
=F\left(            
\frac{1}{2}U^*\begin{pmatrix}
                      A_1 & 0\\
                      0 & 0
                      \end{pmatrix}U
+\frac{1}{2}V^*\begin{pmatrix}
                      A_1 & 0\\
                      0 & 0
                      \end{pmatrix}V,\dots,
\frac{1}{2}U^*\begin{pmatrix}
                      A_k & 0\\
                      0 & 0
                      \end{pmatrix}U
+\frac{1}{2}V^*\begin{pmatrix}
                      A_k & 0\\
                      0 & 0
                      \end{pmatrix}V

\right)\\[3ex]
\displaystyle\le \frac{1}{2} F\left(            
U^*\begin{pmatrix}
                 A_1 & 0\\
                 0 & 0
                 \end{pmatrix}U
,\dots,
U^*\begin{pmatrix}
                A_k & 0\\
                0 & 0
                \end{pmatrix}U\right)
+\frac{1}{2} F\left(            
V^*\begin{pmatrix}
                 A_1 & 0\\
                 0 & 0
                 \end{pmatrix}V
,\dots,
V^*\begin{pmatrix}
                A_k & 0\\
                0 & 0
                \end{pmatrix}V\right)\\[3ex]
=\displaystyle\frac{1}{2}U^*F\left(\begin{pmatrix}
                                 A_1 & 0\\
                                 0 & 0
                                 \end{pmatrix}
,\dots,
\begin{pmatrix}
                A_k & 0\\
                0 & 0
                \end{pmatrix}\right)U
+\frac{1}{2}V^*F\left(\begin{pmatrix}
                                 A_1 & 0\\
                                 0 & 0
                                 \end{pmatrix}
,\dots,
\begin{pmatrix}
                A_k & 0\\
                0 & 0
                \end{pmatrix}\right)V\\[3ex]
\displaystyle=\frac{1}{2}U^*\begin{pmatrix}
                                                     F(A_1,\dots,A_k) & 0\\
                                                     0                          & F(0,\dots,0)
                                                     \end{pmatrix}U
+\frac{1}{2}V^*\begin{pmatrix}
                                                     F(A_1,\dots,A_k) & 0\\
                                                     0                          & F(0,\dots,0)
                                                     \end{pmatrix}V\\[3ex]
\displaystyle\le\frac{1}{2}U^*\begin{pmatrix}
                                                     F(A_1,\dots,A_k) & 0\\
                                                     0                          & 0
                                                     \end{pmatrix}U
+\frac{1}{2}V^*\begin{pmatrix}
                                                     F(A_1,\dots,A_k) & 0\\
                                                     0                          & 0
                                                     \end{pmatrix}V\\[3ex]                            
\displaystyle=\begin{pmatrix}
                                C^*F(A_1,\dots,A_k)C & 0\\
                                0                                  & S F(A_1,\dots,A_k) S
                                \end{pmatrix},
\end{array}
\]
where we used convexity in the first inequality, and in the second inequality used $ F(0,\dots,0)\le 0. $ The first statement now follows.

In order to prove $ (ii) $ we define the map
\[
G(A_1,\dots, A_k)=F(A_1,\dots,A_k)-F(0,\dots,0)\qquad (A_1,\dots,A_k)\in\mathcal D^k.
\]
Unitary invariance of $ F $ implies that $ F(0,\dots,0) $ is a multiple of the unit operator and thus commutes with all projections. Therefore $ G $ is regular and convex with  $ G(0,\dots,0)=0. $ We then define block matrices
\[
C=\begin{pmatrix}
     X & 0\\
     Y & 0
     \end{pmatrix}
\qquad\text{and}\qquad
Z_m=\begin{pmatrix}
                  A_m & 0\\
                  0    & B_m
                  \end{pmatrix},\quad m=1,\dots,k
\]
and notice that
\[
C^*Z_mC=\begin{pmatrix}
                            X^*A_mX+Y^* B_m Y & 0\\
                            0                                  & 0
                            \end{pmatrix}
\]
for $ m=1,\dots,k. $ Finally we use $ (i) $ to obtain 
\[
\begin{array}{l}
\displaystyle
\begin{pmatrix}
G\bigl(X^*A_1X+Y^*B_1Y,\dots,X^*A_kX+Y^*B_kY\Bigr) & 0\\
0                                                                                          & 0
\end{pmatrix}\\[5ex]
\displaystyle
=\begin{pmatrix}
G\bigl(X^*A_1X+Y^*B_1Y,\dots,X^*A_kX+Y^*B_kY\Bigr) & 0\\
0                                                                                          & G(0,\dots,0)
\end{pmatrix}\\[4ex]
=G\bigl(C^*Z_1C,\dots,C^* Z_k C\bigr)\\[2.5ex]
\le C^* G(Z_1,\dots,Z_k) C
=\displaystyle C^*\begin{pmatrix}
                                     G(A_1,\dots,A_k) & 0\\
                                     0                           & G(B_1,\dots,B_k)
                                     \end{pmatrix} C\\[4ex]
\displaystyle=\begin{pmatrix}
                               X^* G(A_1,\dots,A_k)X + Y^* G(B_1,\dots,B_k) Y & 0\\
                               0                                                                             & 0
                               \end{pmatrix}
\end{array}
\]
from which we deduce that
\[
\begin{array}{l}
F\bigl(X^*A_1X+Y^*B_1Y,\dots,X^*A_kX+Y^*B_kY\Bigr)\\[2ex]
=G\bigl(X^*A_1X+Y^*B_1Y,\dots,X^*A_kX+Y^*B_kY\Bigr)+F(0,\dots,0)\\[2ex]
\le X^* G(A_1,\dots,A_k)X + Y^* G(B_1,\dots,B_k) Y+F(0,\dots,0)\\[2ex]
= X^* F(A_1,\dots,A_k)X + Y^* F(B_1,\dots,B_k) Y\\[1ex]
\hskip 8em -\,X^*F(0,\dots,0)X-Y^*F(0,\dots,0)Y+F(0,\dots,0).
\end{array}
\]
Since as above $ F(0,\dots,0)=c\cdot 1 $ for some real constant $ c $ we obtain
\[
\begin{array}{l}
-X^*F(0,\dots,0)X-Y^*F(0,\dots,0)Y+F(0,\dots,0)\\[1.5ex]
=-\,c(X^*X+Y^*Y)+c\cdot 1=0,
\end{array}
\]
and the statement of the theorem follows.
\end{proof}

We shall for $ k=1,2,\dots $ consider the convex domain
\[
\mathcal D_+^k = \{ (A_1,\dots,A_k)\mid A_1,\dots,A_k>0\}
\]
of positive definite and invertible operators acting on the Hilbert space $ \mathcal H. $

\begin{proposition}\label{proposition: transformer equality}
Let $ F $ be a regular map of $ \mathcal D_+^k $ into self-adjoint operators acting on $ \mathcal H. $ We assume that
\begin{enumerate}[(i)]

\item $ F $ is convex

\item $ F(tA_1,\dots,tA_k)=tF(A_1,\dots,A_k) $ \qquad $ t>0, $ \quad
$ (A_1,\dots,A_k)\in\mathcal D_+^k\,. $ 

\end{enumerate}
Then
\[
F(C^*A_1C,\dots,C^*A_kC)=C^*F(A_1,\dots,A_k)C
\]
for any invertible operator $ C $ on $ \mathcal H $ and $ (A_1,\dots,A_k)\in\mathcal D_+^k. $
\end{proposition}

\begin{proof}
Assume first that $ C $ is an invertible contraction on $ \mathcal H. $ Jensen's sub-homogeneous inequality is only available for regular mappings defined in $ \mathcal D^k. $ To $ \varepsilon>0 $ we therefore consider the mapping $ F_\varepsilon\colon\mathcal D^k\to B(\mathcal H) $ by setting
\[
F_\varepsilon(A_1,\dots,A_k)=F(\varepsilon+A_1,\dots,\varepsilon + A_k)-F(\varepsilon,\dots,\varepsilon).
\]
By unitary invariance of $ F $ we realise that $ F(\varepsilon, \dots,\varepsilon) $ is a multiple of the unity. Therefore, $ F_\varepsilon $ is regular and convex with $ F_\varepsilon(0,\dots,0)=0. $

We may thus use Jensen's sub-homogeneous inequality for regular mappings and obtain 
\[
F_\varepsilon(C^*A_1 C,\dots, C^* A_k C)\le C^* F_\varepsilon(A_1,\dots,A_k)C,
\]
where we now restrict $ (A_1,\dots,A_k) $ to the domain $ \mathcal D_+^k $ and rearrange the inequality to
\[
\begin{array}{l}
F(\varepsilon+C^*A_1C,\dots,\varepsilon + C^*A_k C)\\[1.5ex]
\le C^* F(A_1+\varepsilon,\dots,A_k+\varepsilon)C+F(\varepsilon,\dots,\varepsilon)-C^*F(\varepsilon,\dots,\varepsilon) C.
\end{array}
\]
Since $ F $ is positively homogeneous the term $ F(\varepsilon,\dots,\varepsilon)=\varepsilon F(1,\dots,1) $ is vanishing for $ \varepsilon\to 0 $ and we obtain
\begin{equation}\label{Jensen inequality in the corollary}
F(C^*A_1 C,\dots, C^* A_k C)\le C^* F(A_1,\dots,A_k)C
\end{equation}
for invertible $ C. $ Again using homogeneousness we obtain inequality (\ref{Jensen inequality in the corollary}) also for arbitrary invertible $ C. $ Then by repeated application of (\ref{Jensen inequality in the corollary}) we obtain
\[
F(A_1,\dots,A_k)\le {C^*}^{-1}F(C^*A_1 C,\dots,C^*A_k C) C^{-1}\le F(A_1,\dots,A_k)
\]
and the statement follows.
\end{proof}

\section{The perspective of a regular map}

\begin{definition}
Let $ F\colon\mathcal D_+^k\to B(\mathcal H) $ be a regular mapping. The perspective map $ \mathcal P_F $ is the mapping defined in the domain $ \mathcal D_+^{k+1} $ by setting
\[
\mathcal P_F(A_1,\dots,A_k,B)=B^{1/2} F(B^{-1/2}A_1B^{-1/2},\dots,B^{-1/2}A_kB^{-1/2}) B^{1/2}
\]
for positive invertible operators $ A_1,\dots,A_k $ and $ B $ acting on $ \mathcal H. $
\end{definition}

It is a small exercise to prove that the perspective $ \mathcal P_F $ is a regular mapping which is positively homogeneous in the sense that
\[
\mathcal P_F(tA_1,\dots,tA_k,tB)=t \mathcal P_F(A_1,\dots,A_k,B)
\]
for arbitrary $ (A_1,\dots,A_k,B)\in\mathcal D_+^{k+1} $ and real numbers $ t>0. $
The following theorem generalises a result of Effros \cite[Theorem 2.2]{kn:effros:2009:1} for functions of one variable.

\begin{theorem}\label{convexity of perspective} 
The perspective $ \mathcal P_F $ of a convex regular map $ F\colon\mathcal D_k^+\to B(\mathcal H) $ is convex.
\end{theorem}

\begin{proof}
Consider tuples $ (A_1,\dots,A_{k+1}) $ and $ (B_1,\dots,B_{k+1}) $ in $ \mathcal D_+^{k+1} $ and take  $ \lambda\in[0,1]. $ We define the operators
\[
\begin{array}{rl}
C&=\lambda A_{k+1} + (1-\lambda) B_{k+1}\\[1.5ex]
X&=\lambda^{1/2} A_{k+1}^{1/2} C^{-1/2}\\[1.5ex]
Y&=(1-\lambda)^{1/2} B_{k+1}^{1/2} C^{-1/2}
\end{array}
\]
and calculate that
\[
X^*X+Y^*Y=C^{-1/2}\lambda A_{k+1}C^{-1/2}+C^{-1/2} (1-\lambda) B_{k+1} C^{-1/2}=1
\]
and
\[
\begin{array}{l}
X^* A_{k+1}^{-1/2} A_i A_{k+1}^{-1/2} X+Y^* B_{k+1}^{-1/2} B_i B_{k+1}^{-1/2} Y\\[1.5ex]
= C^{-1/2}\lambda^{1/2}A_{k+1}^{1/2} A_{k+1}^{-1/2} A_i A_{k+1}^{-1/2} \lambda^{1/2}  A_{k+1}^{1/2}C^{-1/2}\\[1ex]
\hskip 4em+\,C^{-1/2}(1-\lambda)^{1/2} B_{k+1}^{1/2} B_{k+1}^{-1/2} B_i B_{k+1}^{-1/2} (1-\lambda)^{1/2}  B_{k+1}^{1/2}C^{-1/2}\\[1.5ex]
=\, C^{-1/2}(\lambda A_i + (1-\lambda) B_i) C^{-1/2}
\end{array}
\]
for $ i=1,\dots,k. $ We thus obtain
\[
\begin{array}{l}
\mathcal P_F(\lambda A_1+(1-\lambda) B_1,\dots,\lambda A_{k+1} + (1-\lambda) B_{k+1})\\[1.5ex]
=C^{1/2} F\bigl(C^{-1/2}(\lambda A_1+(1-\lambda) B_1)C^{-1/2},\dots,\\[1ex]
\hskip 6em C^{-1/2}(\lambda A_k+(1-\lambda) B_k) C^{-1/2}\bigr) C^{1/2}\\[1.5ex]
=C^{1/2} F\bigl(X^* A_{k+1}^{-1/2} A_1 A_{k+1}^{-1/2} X+Y^* B_{k+1}^{-1/2} B_1 B_{k+1}^{-1/2} Y,\dots,\\[1ex]
\hskip 6em X^* A_{k+1}^{-1/2} A_k A_{k+1}^{-1/2} X+Y^* B_{k+1}^{-1/2} B_k B_{k+1}^{-1/2} Y\bigr) C^{1/2}\\[2ex]
\le C^{1/2} \bigl(X^* F(A_{k+1}^{-1/2} A_1 A_{k+1}^{-1/2},\dots, A_{k+1}^{-1/2} A_k A_{k+1}^{-1/2}) X\\[1ex] 
\hskip 6em+\, Y^* F(B_{k+1}^{-1/2} B_1 B_{k+1}^{-1/2},\dots, B_{k+1}^{-1/2} B_k B_{k+1}^{-1/2}) Y\bigr) C^{1/2}\\[1.5ex]
= \lambda A_{k+1}^{1/2}  F(A_{k+1}^{-1/2} A_1 A_{k+1}^{-1/2},\dots, A_{k+1}^{-1/2} A_k A_{k+1}^{-1/2}) A_{k+1}^{1/2}\\[1ex]
\hskip 6em +\, (1-\lambda)B_{k+1}^{1/2} F(B_{k+1}^{-1/2} B_1 B_{k+1}^{-1/2},\dots, B_{k+1}^{-1/2} B_k B_{k+1}^{-1/2}) B_{k+1}^{1/2}\\[2ex]
=\lambda\mathcal P_F(A_1,\dots,A_{k+1})+(1-\lambda)\mathcal P_F(B_1,\dots,B_{k+1}),
\end{array}
\]
where we used Jensen's inequality for regular mappings.
\end{proof}

\begin{proposition}\label{proposition: homogeneous map as a perspective}
Let $ F\colon\mathcal D_+^{k+1}\to B(\mathcal H) $ be a convex and positively homogeneous regular mapping. Then $ F $ is the perspective of its restriction $ G $ to $ \mathcal D_+^k $ given by
\[
G(A_1,\dots,A_k)= F(A_1,\dots, A_k,1)
\]
for positive invertible operators $ A_1,\dots,A_k $ acting on $ \mathcal H. $
\end{proposition}

\begin{proof}
Since $ F $ is a convex and positively homogeneous regular mapping we may apply Proposition~\ref{proposition: transformer equality}. Then by setting $ C=A_{k+1}^{-1/2} $ we obtain
\[
\begin{array}{l}
A_{k+1}^{-1/2} F(A_1,\dots,A_k,A_{k+1}) A_{k+1}^{-1/2}\\[1.5ex]
\hskip 9em=F(A_{k+1}^{-1/2}A_1A_{k+1}^{-1/2},\dots,A_{k+1}^{-1/2}A_k A_{k+1}^{-1/2}, 1).
\end{array}
\]
By rearranging this equation we obtain
\[
F(A_1,\dots,A_k,A_{k+1})=A_{k+1}^{1/2} G(A_{k+1}^{-1/2}A_1A_{k+1}^{-1/2},\dots,A_{k+1}^{-1/2}A_k A_{k+1}^{-1/2}) A_{k+1}^{1/2}
\]
which is the statement to be proved.
\end{proof}

The result in the above proposition may be reformulated in the following way: The perspective $ \mathcal P_G $ of a convex regular mapping $ G\colon\mathcal D_+^k\to B(\mathcal H) $ is the unique extension of $ G $ to a positively homogeneous convex regular mapping $ F\colon\mathcal D_+^{k+1}\to B(\mathcal H). $

\section{The construction of geometric means}\label{construction of the geometric mean}

We construct a sequence of multivariate geometric means $ G_1, G_2, \dots $ by the following general procedure.

\begin{enumerate}[(i)]

\item We begin by setting $ G_1(A)=A $ for each positive definite invertible operator $ A. $

\item To each geometric mean $ G_k $ of $ k $ variables we associate an auxiliary mapping $ F_k\colon \mathcal D_+^k\to B(\mathcal H) $ such that
\begin{enumerate}[(a)]

\item $ F_k $ is a regular map,

\item $ F_k $ is concave,

\item $ F_k(t_1,\dots,t_k)=\big(t_1\cdots t_k\bigr)^{1/(k+1)} $ for positive numbers $ t_1,\dots, t_k. $

\end{enumerate}

\item We define the geometric mean $ G_{k+1}\colon\mathcal D_+^{k+1}\to B(\mathcal H) $ of $ k+1 $ variables as the perspective
\[
G_{k+1}(A_1,\dots,A_{k+1})=\mathcal P_{F_k}(A_1,\dots,A_{k+1})
\]
of the auxiliary map $ F_k\,. $

\end{enumerate}

Geometric means defined by this very general procedure are concave and positively homogeneous regular mappings by Theorem~\ref{convexity of perspective} and the preceding remarks. They also satisfy
\begin{equation}\label{geometric mean of commuting operators}
G_k(A_1,\dots,A_k)=\bigl(A_1\cdots A_k\bigr)^{1/k}
\end{equation}
for commuting operators. Indeed, since $ G_k $ is the perspective of $ F_{k-1} $ and this map satisfies $ (c) $ in condition $ (ii), $ we obtain $ G_k(t_1,\dots,t_k)=\bigl(t_1\cdots t_k\bigr)^{1/k} $ for positive numbers. Equality (\ref{geometric mean of commuting operators}) then follows since $ G_k $ is regular.
The geometric mean of two variables
\begin{equation}
G_2(A_1,A_2)=A_2^{1/2}\bigl(A_2^{-1/2}A_1A_2^{-1/2}\bigr)^{1/2}A_2^{1/2}
\end{equation}
coincides with the geometric mean of two variables $ A_1\#A_2 $ introduced by Pusz and Woronowicz. This is so since $ G_2 $ is the perspective of $ F_1 $ and $ F_1(A)=A^{1/2}. $ The last statement is obtained since $ F_1 $ is a regular mapping and satisfies $ F_1(t)=t^{1/2} $ for positive numbers by $ (c) $ in condition $ (ii). $ 

There are many ways to associate the auxiliary map $ F_k $ in the above procedure, so we should not in general expect much similarity between the geometric means for different number of variables. 

\subsection{The inductive geometric mean}

We define the auxiliary mapping $ F_k\colon \mathcal D_+^k\to B(\mathcal H) $ by setting
\[
F_k(A_1,\dots,A_k)=G_k(A_1,\dots,A_k)^{k/(k+1)}
\]
for $ k=1,2,\dots. $

\begin{theorem}\label{inductive geometric mean: defining properties}
The means $ G_k $ constructed in section \ref{construction of the geometric mean} then have the following properties:

\begin{enumerate}[(i)]

\item $ G_k\colon\mathcal D_+^k\to B(\mathcal H)_+ $ is a regular map for each $ k=1,2,\dots. $ 

\item $ G_k(tA_1,\dots, tA_k)=t G_k(A_1,\dots,A_k) $ for $ t>0, $ $ (A_1,\dots,A_k)\in\mathcal D_+^k $ and $ k=1,2,\dots. $

\item $ G_k\colon\mathcal D_+^k\to B(\mathcal H) $ is  concave for each $ k=1,2,\dots. $

\item $ G_{k+1}(A_1,\dots,A_k,1)=G_k(A_1,\dots,A_k)^{k/(k+1)} $ for $ (A_1,\dots,A_k)\in\mathcal D_+^k $ and $ k=1,2,\dots. $ 

\end{enumerate}

Any sequence of mappings $ \tilde G_k $ beginning with $ \tilde G_1(A)=A $ and satisfying the above conditions coincide with the means $ G_k $ for $ k=1,2,\dots. $

\end{theorem}

\begin{proof} Each map $ G_k $ is for $ k=2,3,\dots $ the perspective of a regular map and this implies $ (i) $ and $ (ii). $ The assertion of concavity for $ G_1 $ is immediate. 
Suppose now $ G_k $ is concave for some $ k. $  Since the map $ t\to t^p $ is both operator monotone and operator concave for $ 0\le p\le 1, $ we realise that the auxiliary mapping
\[
F_k(A_1,\dots,A_k)=G_k(A_1,\dots,A_k)^{k/(k+1)}
\]
is concave, and since $ G_{k+1} $ is the perspective of $ F_k $ we then obtain by Theorem \ref{convexity of perspective} that also $ G_{k+1} $ is concave. Since the first map $ G_1 $ is concave we have thus proved by induction that $ G_k $ is concave for all $ k=1,2.\dots. $
The last property $ (iv) $ follows since $ G_{k+1} $ is the perspective  of $ G_k^{k/(k+1)}. $

Let finally $ \tilde G_k $ be a sequence of mappings satisfying $ (i) $ to $ (iv). $ Since each $ \tilde G_{k+1} $ is concave and homogeneous it follows by Proposition~\ref{proposition: homogeneous map as a perspective} that $ \tilde G_{k+1} $ is the perspective of its restriction  $ \tilde G_{k+1}(A_1,\dots,A_k,1). $ Because of $ (iv) $ we then realise that $ \tilde G_{k+1} $ is the perspective of the map
\[
\tilde F_k(A_1,\dots,A_k)=\tilde G_k(A_1,\dots,A_k)^{k/(k+1)}
\]
constructed from $ \tilde G_k. $ The $ \tilde G_k $ mappings are thus constructed by the same algorithm as the mappings $ G_k $ for every $ k\ge 2, $ and since $ \tilde G_1=G_1 $ they must all coincide. \end{proof}

In addition to the properties listed in the above theorem the means $ G_k $ enjoy a number of other properties that we list below.

\begin{theorem}\label{geometric mean: additional properties}
The means $ G_k $ constructed in section \ref{construction of the geometric mean} have the following additional properties:

\begin{enumerate}[(i)]

\item The means $ G_k $ are increasing in each variable for $ k=1,2\dots. $

\item The means $ G_k $ are congruence invariant. For any invertible operator $ C $ on $ \mathcal H $ the identity
\[
G_k(C^*A_1C,\dots,C^*A_kC)=C^* G_k(A_1,\dots,A_k) C
\]
holds for $ (A_1,\dots,A_k)\in\mathcal D_+^k $ and $ k=1,2,\dots. $

\item The means $ G_k $ are jointly homogeneous in the sense that
\[
G_k(t_1A_1,\dots,t_k A_k)=(t_1\cdots t_k)^{1/k} G_k(A_1,\dots,A_k)
\]
for scalars $ t_1,\dots,t_k>0, $  operators $ (A_1,\dots,A_k)\in\mathcal D_+^k $ and $ k=1,2,\dots. $

\item The means $ G_k $ are self-dual in the sense that
\[
G_k(A_1^{-1},\dots,A_k^{-1})=G_k(A_1,\dots,A_k)^{-1}
\]
for $ (A_1,\dots,A_k)\in\mathcal D_+^k $ and $ k=1,2,\dots. $

\item When restricted to positive definite matrices the determinant identity
\[
\det G_k(A_1,\dots,A_k)=(\det A_1\cdots\det A_k)^{1/k}
\]
holds for $ k=1,2\dots. $

\end{enumerate}

\end{theorem}

\begin{proof}
The first property follows by the following standard argument for positive concave mappings. Consider  positive definite invertible operators $ A_m \le B_m $ for $ m=1,\dots,k. $ By first assuming that the difference $ B_m-A_m $ is invertible we may take $ \lambda\in (0,1) $ and write
\[
\lambda B_m=\lambda A_m+(1-\lambda) C_m \qquad m=1,\dots,k,
\]
where $ C_m=\lambda(1-\lambda)^{-1}(B_m-A_m) $ is positive definite and invertible.
By using concavity we then obtain
\[
\begin{array}{rl}
G_k(\lambda B_1,\dots,\lambda B_k)
&\ge \lambda G(A_1,\dots,A_k)+(1-\lambda) G_k(C_1,\dots,C_k)\\[1.5ex]
&\ge \lambda G(A_1,\dots,A_k).
\end{array}
\]
Letting $ \lambda\to 1 $ we obtain $ G_k(B_1,\dots,B_k)\ge G_k(A_1,\dots,A_k) $ by continuity. In the general case we choose $ 0<\mu<1 $ such that
\[
\mu A_m<A_m\le B_m\qquad m=1,\dots,k
\]
and obtain $ G_k(\mu A_1,\dots,\mu A_k)\le G_k(B_1,\dots,B_k). $ By letting $ \mu\to 1 $ we then obtain $ G_k(A_1,\dots, A_k)\le G_k(B_1,\dots,B_k) $ which shows $ (i). $

Since $ G_k $ is concave and homogeneous we obtain $ (ii) $ from Proposition~\ref{proposition: transformer equality}. 

Property $ (iii) $ is immediate for $ k=1 $ and $ k=2. $ Suppose the property is verified for $ k, $ then
\[
\begin{array}{l}
G_{k+1}(t_1 A_1,\dots,t_k A_k, t_{k+1}A_{k+1})\\[1.5ex]
=t_{k+1}A_{k+1}^{1/2}F_k\bigl(t_1t_{k+1}^{-1} A_{k+1}^{-1/2}A_1A_{k+1}^{-1/2},\dots,t_k t_{k+1}^{-1}A_{k+1}^{-1/2} A_k A_{k+1}^{-1/2}\bigr)A_{k+1}^{1/2}\\[1.5ex]
=t_{k+1}A_{k+1}^{1/2}G_k\bigl(t_1t_{k+1}^{-1} A_{k+1}^{-1/2}A_1A_{k+1}^{-1/2},\dots,t_k t_{k+1}^{-1}A_{k+1}^{-1/2} A_k A_{k+1}^{-1/2}\bigr)^{k/(k+1)} A_{k+1}^{1/2}\,.
\end{array}
\]
By using the induction assumption we obtain
\[
\begin{array}{l}
G_{k+1}(t_1 A_1,\dots,t_k A_k, t_{k+1}A_{k+1})\\[1.5ex]
=t_{k+1}(t_{k+1}^{-1} t_1^{1/k}\cdots t_k^{1/k})^{k/(k+1)} G_{k+1}(A_1,\dots,A_k,A_{k+1})\\[1.5ex]
=(t_1\cdots t_k t_{k+1})^{1/(k+1)} G_{k+1}(A_1,\dots,A_k,A_{k+1})
\end{array}
\]
which shows $ (iii). $

Property $ (iv) $ is immediate for $ k=1 $ and $ k=2. $ Suppose the property is verified for $ k, $ then
\[
\begin{array}{l}
G_{k+1}(A_1^{-1},\dots,A_k^{-1},A_{k+1}^{-1})\\[1.5ex]
=A_{k+1}^{-1/2} F_k\bigl(A_{k+1}^{1/2}A_1^{-1} A_{k+1}^{1/2},\dots, A_{k+1}^{1/2} A_k^{-1} A_{k+1}^{1/2}\bigr) A_{k+1}^{-1/2}\\[1.5ex]
=A_{k+1}^{-1/2} G_k\bigl(A_{k+1}^{1/2}A_1^{-1} A_{k+1}^{1/2},\dots, A_{k+1}^{1/2} A_k^{-1} A_{k+1}^{1/2}\bigr)^{k/(k+1)} A_{k+1}^{-1/2}\,.
\end{array}
\]
By using the induction assumption we obtain
\[
\begin{array}{l}
G_{k+1}(A_1^{-1},\dots,A_k^{-1},A_{k+1}^{-1})\\[1ex]
=A_{k+1}^{-1/2} G_k\bigl(A_{k+1}^{-1/2}A_1 A_{k+1}^{-1/2},\dots, A_{k+1}^{-1/2} A_k A_{k+1}^{-1/2}\bigr)^{-k/(k+1)} A_{k+1}^{-1/2}\\[1.5ex]
=\bigl(A_{k+1}^{1/2} G_k\bigl(A_{k+1}^{-1/2}A_1 A_{k+1}^{-1/2},\dots, A_{k+1}^{-1/2} A_k A_{k+1}^{-1/2}\bigr)^{k/(k+1)} A_{k+1}^{1/2}\bigr)^{-1}\\[1.5ex]
=G_{k+1}(A_1,\dots,A_k,A_{k+1})^{-1}
\end{array}
\]
which shows $ (iv). $

Notice that since $ \det A=\exp(\tr\log A) $ for positive definite $ A, $ we have $ \det A^p=(\det A)^p $ for all real exponents $ p. $ Property $ (v) $ is easy to calculate for $ k=1 $ and $ k=2. $ Suppose the property is verified for $ k. $ Since as above
\[
\begin{array}{l}
G_{k+1}(A_1,\dots,A_k,A_{k+1})\\[1.5ex]
=A_{k+1}^{1/2} G_k\bigl(A_{k+1}^{-1/2}A_1 A_{k+1}^{-1/2},\dots, A_{k+1}^{-1/2} A_k A_{k+1}^{-1/2}\bigr)^{k/(k+1)} A_{k+1}^{1/2}
\end{array}
\]
we obtain
\[
\begin{array}{l}
\det G_{k+1}(A_1,\dots,A_k,A_{k+1})\\[1.5ex]
=\det A_{k+1}\bigl(\det A_{k+1}^{-1}\det A_1\cdots \det A_{k+1}^{-1}\det A_k\bigr)^{1/(k+1)}\\[1.5ex]
=\bigl(\det A_1\cdots\det A_k\cdot\det A_{k+1}\bigr)^{1/k+1}
\end{array}
\]
which shows $ (v). $
\end{proof}

\begin{theorem}\label{HGA mean inequality}
The geometric means $ G_k $ are for $ k=1,2,\dots $ bounded between the symmetric harmonic and arithmetic means. That is,
\[
\frac{k}{A_1^{-1}+\cdots+A_k^{-1}}\le G_k(A_1,\dots,A_k)\le\frac{A_1+\cdots+A_k}{k}
\]
for arbitrary $ (A_1,\dots,A_k)\in\mathcal D_+^k $ and $ k=1,2,\dots. $
\end{theorem}

\begin{proof}
The upper bound holds with equality for $ k=1. $ Suppose that we have verified the inequality for $ k. $ Since by classical analysis
\[
X^{k/(k+1)}\le 1+\frac{k}{k+1} (X-1)
\]
for positive definite $ X, $ we obtain
\[
\begin{array}{l}\displaystyle
F_k(A_1,\dots,A_k)=G_k(A_1,\dots,A_k)^{k/(k+1)}\le 1+\frac{k}{k+1}(G_k(A_1,\dots,A_k)-1)\\[2ex]
\le\displaystyle 1+\frac{k}{k+1}\Bigl(\frac{A_1+\cdots+A_k}{k} -1\Bigr)
=\frac{A_1+\cdots+A_k+1}{k+1}\,.
\end{array}
\]
By taking perspectives we now obtain
\[
\begin{array}{l}
G_{k+1}(A_1,\dots,A_k,B)=\mathcal P_{F_k}(A_1,\dots,A_k,B)\\[2ex]
=B^{1/2}F_k(B^{-1/2}A_1 B^{-1/2},\dots,B^{-1/2}A_k B^{-1/2})B^{1/2}\\[2ex]
\le\displaystyle B^{1/2}\frac{B^{-1/2}A_1 B^{-1/2}+\cdots+B^{-1/2}A_k B^{-1/2} + 1}{k+1} B^{1/2}
=\frac{A_1+\cdots A_k+B}{k+1}
\end{array}
\]
which proves the upper bound by induction. We next use the upper bound to obtain
\[
G_k(A_1^{-1},\dots,A_k^{-1})\le\frac{A_1^{-1}+\cdots+A_k^{-1}}{k}\,.
\]
By inversion we then obtain
\[
\frac{k}{A_1^{-1}+\cdots+A_k^{-1}}\le G_k(A_1^{-1},\dots,A_k^{-1})^{-1}=G_k(A_1,\dots,A_k),
\]
where we in the last equation used self-duality of the geometric mean, cf. property $ (iv) $ in Theorem~\ref{geometric mean: additional properties}.
\end{proof}

The means studied in this section are known in the literature as the inductive means of Sagae and Tanabe \cite{kn:sagae:1994}. By considering the power mean
\[
A\#_tB=B^{1/2}\bigl(A^{-1/2}BA^{-1/2}\bigr)^t B^{1/2}\qquad 0\le t\le 1
\]
they established the recursive relation by setting
\[
G_{k+1}(A_1,\dots,A_{k+1})=G_k(A_1,\dots,A_k)\#_{k/(k+1)} A_{k+1}.
\]
The authors did not study the general properties of these means but established the harmonic-geometric-arithmetic mean inequality of Theorem~\ref{HGA mean inequality}. It is possible to prove the crucial concavity property $ (iii) $ in Theorem~\ref{inductive geometric mean: defining properties} by induction. It can be done without the general theory of perspectives of regular operator mappings, and it only requires the properties of an operator mean of two variables as studied by Kubo and Ando \cite{kn:kubo:1980}. However, this is a special situation that only applies to the inductive means.

\subsection{Variant geometric means}

The inductive geometric means are uniquely specified within the general framework discussed in this paper by choosing the updating condition (\ref{updating condition}), cf. property $ (iv) $ in Theorem~\ref{inductive geometric mean: defining properties}. We may instead construct geometric means satisfying updating condition (\ref{variant updating condition}) by choosing the auxiliary map
\[
F_k(A_1,\dots,A_k)=G_k(A_1^{k/(k+1)},\dots,A_k^{k/(k+1)})
\]
for $ k=1,2\dots. $   It is a small exercise to realise that these means satisfy all of the properties listed in Theorem~\ref{inductive geometric mean: defining properties}, Theorem \ref{geometric mean: additional properties}, and Theorem~\ref{HGA mean inequality} with the only exception that condition $ (iv) $ in Theorem~\ref{inductive geometric mean: defining properties} is replaced by updating condition  (\ref{variant updating condition}). Concavity of these means cannot be reduced to concavity of operator means of two variables but relies on the general theory of regular operator mappings and Theorem \ref{convexity of perspective}.

\subsection{The Karcher means}

The Karcher mean $ \Lambda_k(A_1,\dots,A_k) $ of $ k $ positive definite invertible operator variables is defined as the unique positive definite solution to the equation
\begin{equation}\label{definition of Karcher mean}
\sum_{i=1}^k \log\bigl(X^{1/2} A_i X^{1/2}\bigr)=0,
\end{equation}
and it enjoys all of the attractive properties of an operator mean listed by Ando, Li, and Mathias, cf. \cite{kn:Lawson:2014}. The defining equation (\ref{definition of Karcher mean}) immediately implies that the Karcher mean $ \Lambda_k\colon\mathcal D_+^k\to B(\mathcal H) $ is a regular operator mapping, and it may therefore be understood within the general framework discussed in this paper by choosing the auxiliary map
\[
F_k(A_1,\dots,A_k)=\Lambda_{k+1}(A_1,\dots,A_k,1).
\]
The problem, however, is that we do not have any explicit expression of $ F_k $ in terms of $ \Lambda_k\,. $

{\small
 

\vfill

\noindent Frank Hansen: Institute for Excellence in Higher Education, Tohoku University, Japan.\\
Email: frank.hansen@m.tohoku.ac.jp.
      }

\end{document}